\documentclass[10pt,a4paper]{amsart}
\usepackage{fix-cm}
\usepackage{amsmath,amssymb,amsthm}
\usepackage{microtype}
\usepackage{cite}
\usepackage[hidelinks]{hyperref}
\usepackage{enumitem}
\usepackage{graphicx}
\usepackage{pgf}

\newtheorem{theorem}{Theorem}[section]
\newtheorem{lemma}[theorem]{Lemma}
\newtheorem{corollary}[theorem]{Corollary}
\newtheorem{proposition}[theorem]{Proposition}
\theoremstyle{definition}
\newtheorem{definition}[theorem]{Definition}
\newtheorem{example}[theorem]{Example}
\theoremstyle{remark}
\newtheorem{remark}[theorem]{Remark}
\numberwithin{equation}{section}

\title[Finite-Step Characterizations of Discrete General Helices]
{Intrinsic Finite-Step Characterizations of Discrete General Helices}

\author{Dae Won Yoon}
\address{Department of Mathematics Education and RINS, Gyeongsang National University,
Jinju 52828, Republic of Korea}
\email{dwyoon@gnu.ac.kr}

\author{Chul Woo Lee}
\address{Department of Mathematics, Kyungpook National University,
Daegu 41566, Republic of Korea}
\email{mathisu@knu.ac.kr}

\author{Jae Won Lee}
\address{Department of Mathematics Education and RINS, Gyeongsang National University,
Jinju 52828, Republic of Korea}
\email{leejaew@gnu.ac.kr}
\thanks{Corresponding author: Jae Won Lee (\texttt{leejaew@gnu.ac.kr}).}

\subjclass[2000]{Primary 53A04; Secondary 53A70}
\keywords{Discrete general helix; polygonal space curve; discrete Frenet frame; turning angle; signed torsion angle; tangent indicatrix}

\hypersetup{
  pdftitle={Intrinsic Finite-Step Characterizations of Discrete General Helices},
  pdfauthor={Dae Won Yoon, Chul Woo Lee, Jae Won Lee}
}

\begin{document}

\begin{abstract}
We study polygonal general helices in Euclidean three-space using the turning and signed torsion angles of the discrete Frenet frame. For helices whose axis is not orthogonal to the edge tangents, we prove that the global constant-angle condition is equivalent to the existence of a conserved Frenet-frame vector. On the generic branch, elimination of the auxiliary coefficient yields an intrinsic finite-step compatibility relation involving three consecutive turning angles and two consecutive torsion angles. Complementary phase and linear-subspace formulations cover the antipodal-binormal case. These characterizations reconstruct the helical axis and the helix angle and yield a sharp bound for each turning angle. We also give a spherical formulation through the tangent indicatrix: its vertices lie on a plane section of the unit sphere, which is a small circle in the non-orthogonal case, whereas the orthogonal case is exactly planar. A nonconstant Frenet-data example illustrates the criterion. Finally, for uniform chordal sampling of a smooth curve, the discrete Lancret-type quotient and the reconstructed helical direction converge with second-order accuracy.
\end{abstract}

\maketitle

\section{Introduction}

A smooth general helix in Euclidean three-space is a regular curve whose unit tangent makes a constant angle with a fixed direction. On an interval where $\kappa\tau\ne0$, Lancret's theorem says that this is equivalent to the constancy of $\tau/\kappa$, or equivalently $\kappa/\tau$ \cite{Struik1988}. There are also versions of Lancret's theorem in other spaces. For example, \c{C}ift\c{c}i considered general helices in three-dimensional Lie groups with a bi-invariant metric \cite{Ciftci2009}.

There are several ways to define curvature and torsion for polygonal curves. Frenet-frame methods, cross-ratio methods, and integrable discrete curve models are all used in the literature \cite{Carroll2014,MullerVaxman2021,Hoffmann2026,Kaji2026}. M\"uller and Vaxman define discrete curvature and torsion by cross-ratios and study their smooth limits \cite{MullerVaxman2021}. Kaji, Kajiwara and Shigetomi use an $SO(3)$ discrete Frenet evolution with signed curvature and torsion angles in their work on kaleidocycles \cite{Kaji2026}. Tellier, Douthe, Hauswirth and Baverel use a related notion of a discrete general helix, namely a polyline whose segments make a constant angle with a reference plane \cite{Tellier2019}. If $U$ is a unit normal to that plane, their unsigned condition is
\[
|T_i\cdot U|=\mathrm{constant}.
\]
The condition used here, $T_i\cdot U=c$, is stronger unless the signs of $T_i\cdot U$ are chosen consistently along the polygon. For example, a tangent sequence that alternates between the two small circles $X\cdot U=\pm c$ satisfies the unsigned condition but does not satisfy Definition~\ref{def:discrete-general-helix} for that $U$.

In this paper we use the usual polygonal Frenet data: the turning angles $\theta_i$ and the signed torsion angles $\phi_i$. We ask when the condition
\[
T_i\cdot U=c
\]
can be checked from these angles alone. For the case $c\ne0$, put
\[
x_i=\tan\frac{\theta_i}{2},
\qquad
Q_i=T_i-x_iN_i+y_iB_i.
\]
We show that the constant-angle condition is equivalent to the existence of coefficients $y_i$ for which $Q_i$ is independent of $i$. If $\phi_i\notin\{0,\pi\}$, the coefficients can be eliminated and one obtains
\[
\frac{x_{i+2}+x_{i+1}\cos\phi_{i+1}}{\sin\phi_{i+1}}
=
\frac{x_i+x_{i+1}\cos\phi_i}{\sin\phi_i}.
\]
The undivided equations for $Q_i$ also cover the case $\phi_i=\pi$.

We first prove the conserved-vector characterization of the non-orthogonal case. On the branch $\phi_i\notin\{0,\pi\}$, eliminating the auxiliary coefficients gives $n-3$ scalar finite-step equations. For the full range $\phi_i\in(-\pi,\pi]$, we use a phase description and a linear criterion. In the linear form, for fixed torsion data the vector
\[
x=(x_1,\ldots,x_{n-1})^{\mathsf T}
\]
belongs to the span of two vectors determined by the torsion angles. The formulas also determine the helical direction and angle and give $\theta_i\le2\alpha$. We then study the smooth limit. Under uniform arclength sampling, the discrete quotient converges to $\kappa/\tau$ with order two, while its normalized finite difference converges to $(\kappa/\tau)'$.

The results below use this Frenet-angle convention. In the cited literature, we have not found the same finite-step criterion or the same full linear criterion for variable polygonal Frenet data. The discrete Frenet frame itself, and the extrinsic definition of a general helix, are not new. Some background on discrete curves and discrete differential geometry is given in \cite{BobenkoSuris2008,Hoffmann2009,Carroll2014}.

Section~2 fixes the Frenet conventions. Section~3 gives the spherical characterization and the conserved-vector equations. Section~4 proves the finite-step, phase, and linear characterizations, gives reconstruction formulas and examples, and studies which prescribed torsion-angle data can occur. Section~5 treats the planar case and the smooth limit. The last section collects the main conclusions.

\section{Discrete Frenet data}

Let
\[
P=(P_0,P_1,\ldots,P_n)\subset\mathbb R^3,\qquad n\ge 3,
\]
be a polygonal curve with \(P_{i+1}\neq P_i\). Its unit edge tangent is
\[
T_i=\frac{P_{i+1}-P_i}{\|P_{i+1}-P_i\|},
\qquad 0\le i\le n-1.
\]

For \(1\le i\le n-1\), assume that \(T_{i-1}\) and \(T_i\) are neither
parallel nor antiparallel, and define
\[
0<\theta_i<\pi,\qquad T_{i-1}\cdot T_i=\cos\theta_i.
\]
For these interior indices $1\le i\le n-1$, the discrete binormal and normal are
\[
B_i=\frac{T_{i-1}\times T_i}{\|T_{i-1}\times T_i\|},
\qquad
N_i=B_i\times T_i.
\]
Thus the Frenet frame \((T_i,N_i,B_i)\) is defined for $1\le i\le n-1$, is positively oriented and orthonormal, and
\begin{equation}
T_{i-1}=\cos\theta_i\,T_i-\sin\theta_i\,N_i.
\label{eq:backward-tangent}
\end{equation}

For \(1\le i\le n-2\), the signed torsion angle
\(\phi_i\in(-\pi,\pi]\) is uniquely fixed by the convention
\begin{equation}
B_{i+1}=-\sin\phi_i\,N_i+\cos\phi_i\,B_i.
\label{eq:torsion-convention}
\end{equation}
With this convention,
\begin{align}
T_{i+1}
&=\cos\theta_{i+1}T_i
+\sin\theta_{i+1}\cos\phi_i\,N_i
+\sin\theta_{i+1}\sin\phi_i\,B_i,
\label{eq:T-transition}\\
N_{i+1}
&=-\sin\theta_{i+1}T_i
+\cos\theta_{i+1}\cos\phi_i\,N_i
+\cos\theta_{i+1}\sin\phi_i\,B_i.
\label{eq:N-transition}
\end{align}
Equivalently,
\[
(T_{i+1},N_{i+1},B_{i+1})=(T_i,N_i,B_i)F_i,
\]
where
\[
F_i=
\begin{pmatrix}
\cos\theta_{i+1}&-\sin\theta_{i+1}&0\\
\sin\theta_{i+1}\cos\phi_i&
\cos\theta_{i+1}\cos\phi_i&
-\sin\phi_i\\
\sin\theta_{i+1}\sin\phi_i&
\cos\theta_{i+1}\sin\phi_i&
\cos\phi_i
\end{pmatrix}\in SO(3).
\]
The endpoint value $\phi_i=\pi$ occurs exactly when $B_{i+1}=-B_i$.  For general Frenet data this means that $T_{i-1},T_i,T_{i+1}$ are coplanar and that the oriented binormal reverses; it does not by itself force $T_{i+1}=T_{i-1}$.  If, in addition, the constant-angle condition holds with $c\ne0$, the compatibility equations (see Lemma~\ref{lem:Q-transition}) force $\theta_{i+1}=\theta_i$, and then $T_{i+1}=T_{i-1}$.  Thus $\phi_i=\pi$ is allowed in our convention.

The transfer-matrix description of discrete Frenet frames is standard; see, for example, Hu--Lundgren--Niemi~\cite{HuLundgrenNiemi2011} and the discussion in \cite{Hoffmann2026}. In the convention used here, positive edge lengths together with the turning angles and signed torsion angles determine the polygon, up to an orientation-preserving Euclidean motion, by iterating the matrices $F_i$ and then summing the edge vectors. The orientation convention is reproduced by the reconstructed polygon, because \eqref{eq:T-transition} and \eqref{eq:torsion-convention} give
\[
T_i\times T_{i+1}
=\sin\theta_{i+1}
\bigl(\cos\phi_i\,B_i-\sin\phi_i\,N_i\bigr)
=\sin\theta_{i+1}B_{i+1}.
\]
Since $0<\theta_{i+1}<\pi$, normalizing the cross product gives the prescribed oriented binormal $B_{i+1}$. Conversely, every polygon satisfying the standing assumptions gives such data.

\begin{example}\label{ex:frame-basic}
Take
\[
T_0=\left(\frac45,0,\frac35\right),
\qquad
T_1=\left(\frac{16}{25},\frac{12}{25},\frac35\right).
\]
Both vectors are unit vectors and
\[
T_0\cdot T_1=\frac{109}{125}.
\]
Hence
\[
\theta_1=\arccos\frac{109}{125}\in(0,\pi),
\]
so the discrete Frenet frame is well defined at this vertex; see Figure~\ref{fig:frenet-frame}.
\end{example}

\begin{figure}[!ht]
\centering
\resizebox{0.95\textwidth}{!}{\input{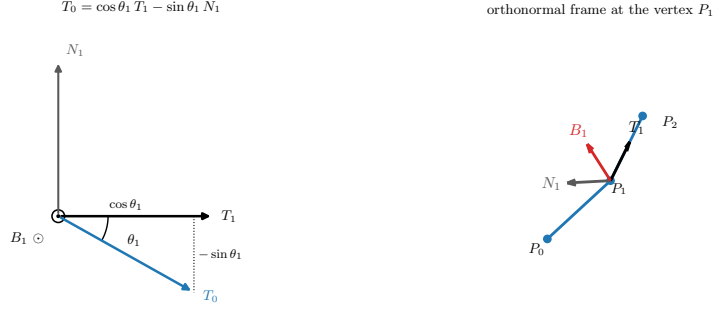}}
\caption{The discrete Frenet frame of Example~\ref{ex:frame-basic}. Left: inside the plane spanned by $T_0$ and $T_1$, equation~\eqref{eq:backward-tangent} reads $T_0=\cos\theta_1\,T_1-\sin\theta_1\,N_1$ with $\theta_1=\arccos(109/125)\approx0.5115$, and the binormal $B_1$ points out of the page. Right: the same data in space at the vertex $P_1$ of a polygon with unit edges, showing the positively oriented orthonormal frame $(T_1,N_1,B_1)$.}
\label{fig:frenet-frame}
\end{figure}

\section{General helices and conserved vectors}

\begin{definition}\label{def:discrete-general-helix}
A polygonal curve is called a discrete general helix, in the tangent-angle sense used here, if there exist a fixed unit vector \(U\) and a constant \(c\) such that
\[
T_i\cdot U=c
\]
for every edge tangent $T_i$, $0\le i\le n-1$. We call the fixed unit vector $U$ the \emph{helical direction}; equivalently, the one-dimensional subspace $\mathbb R U$ is the axis direction. This terminology does not select a distinguished affine line in $\mathbb R^3$.
\end{definition}

\begin{proposition}[Spherical tangent-indicatrix characterization]\label{prop:spherical-characterization}
Let \(T_0,\ldots,T_{n-1}\in S^2\) be the unit edge tangents of a polygon.
Then the polygon is a discrete general helix if and only if its tangent
indicatrix is contained in the intersection
\[
S^2\cap\{X\in\mathbb R^3:X\cdot U=c\}
\]
for some unit vector \(U\) and some \(c\in[-1,1]\).  If \(0<|c|<1\),
this intersection is a small circle on \(S^2\); if \(c=0\), it is a
great circle; and if $|c|=1$, it consists of the single point $\{cU\}$.  Under the standing assumptions $n\ge3$ and $0<\theta_i<\pi$, the last case cannot occur.
\end{proposition}

\begin{proof}
By definition, the polygon is a discrete general helix when
there are a fixed unit vector \(U\) and a constant \(c\) such that
\(T_i\cdot U=c\) for every \(i\).  Since every \(T_i\) belongs to \(S^2\),
this is exactly the stated containment.  Conversely, containment in such
an intersection gives \(T_i\cdot U=c\) for every tangent.  The geometric
description of the intersection follows from the distance \(|c|\) of the
cutting plane from the origin.  When $|c|=1$, equality in Cauchy--Schwarz
forces every tangent in the intersection to equal $\operatorname{sgn}(c)U$;
this is incompatible with $0<\theta_i<\pi$ for consecutive tangents.
\end{proof}

\begin{example}\label{ex:indicatrix}
For the tangents used later in Example~\ref{ex:nonconstant-helix},
\[
T_i\cdot(0,0,1)=\frac35.
\]
Hence their tangent indicatrix lies on
\[
S^2\cap\{z=3/5\},
\]
a small circle of Euclidean radius \(4/5\).  This gives a direct
spherical visualization of the fixed-direction condition; see Figure~\ref{fig:tangent-indicatrix}.
\end{example}

\begin{example}[An example with equally spaced azimuths]\label{ex:symmetric-visual-helix}
Let $U=(0,0,1)$ and, for $0\le k\le11$, set
\[
T_k=\left(\frac45\cos\frac{k\pi}{3},\frac45\sin\frac{k\pi}{3},\frac35\right),
\qquad P_0=0,\qquad P_{k+1}=P_k+T_k.
\]
Then $T_k\cdot U=3/5$ for every $k$, so this is a discrete general helix. Moreover the horizontal components of $T_0,\ldots,T_5$ sum to zero, and hence
\[
P_{k+6}=P_k+\left(0,0,\frac{18}{5}\right).
\]
Thus the polygon repeats after a vertical translation and has the familiar screw-like appearance shown in the right panel of Figure~\ref{fig:tangent-indicatrix}. Its Frenet data are constant:
\[
\theta_k=\arccos\frac{17}{25}\approx0.823033692,
\qquad
\phi_k\approx0.666946345.
\]
This is a symmetric example. A general helix in the tangent-angle sense need not have this rotational symmetry; see Examples~\ref{ex:pi-helix} and~\ref{ex:nonconstant-helix}.
\end{example}

\begin{figure}[!ht]
\centering
\resizebox{0.97\textwidth}{!}{\input{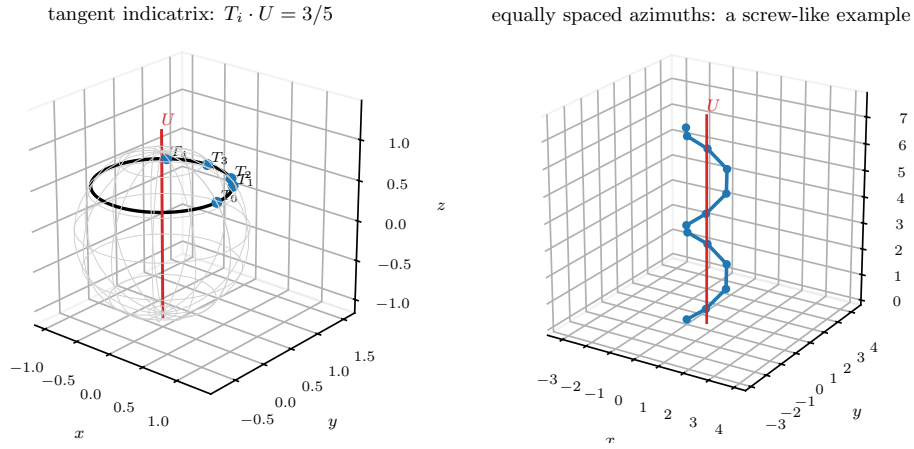}}
\caption{Two views of the condition $T_i\cdot U=3/5$. Left: the tangent indicatrix of Example~\ref{ex:nonconstant-helix} lies on the small circle $S^2\cap\{X\cdot U=3/5\}$ of Euclidean radius $4/5$. Right: the equally spaced-azimuth example of Example~\ref{ex:symmetric-visual-helix}; its polygon visibly winds upward and repeats after a translation parallel to $U$. The right panel is a symmetric special case; the definition only requires a fixed tangent angle.}
\label{fig:tangent-indicatrix}
\end{figure}

We first treat the spatial case \(c\neq0\). Put
\begin{equation}
x_i=\tan\frac{\theta_i}{2}.
\label{eq:x-def}
\end{equation}
Since \(0<\theta_i<\pi\), one has \(x_i>0\).

Suppose \(T_i\cdot U=c\neq0\). Write
\[
U=a_iT_i+b_iN_i+d_iB_i.
\]
Then \(a_i=c\).  Since \(U\cdot T_{i-1}=c\), equation
\eqref{eq:backward-tangent} gives
\[
c=c\cos\theta_i-b_i\sin\theta_i.
\]
Thus
\[
b_i=-c\frac{1-\cos\theta_i}{\sin\theta_i}
=-c\tan\frac{\theta_i}{2}=-cx_i.
\]
Consequently,
\begin{equation}
U=cQ_i,\qquad
Q_i=T_i-x_iN_i+y_iB_i,
\label{eq:Qi}
\end{equation}
where \(y_i=d_i/c\), and hence $Q_i$ is defined for the interior Frenet indices $1\le i\le n-1$.

\begin{lemma}[Frenet transition for the conserved-vector ansatz]\label{lem:Q-transition}
For arbitrary real coefficients $y_i$ and $x_i=\tan(\theta_i/2)$, set
\[
Q_i=T_i-x_iN_i+y_iB_i,\qquad 1\le i\le n-1.
\]
For every $1\le i\le n-2$, the Frenet transition gives
\begin{equation}
\begin{aligned}
Q_{i+1}=T_i
&+\bigl(x_{i+1}\cos\phi_i-y_{i+1}\sin\phi_i\bigr)N_i\\
&+\bigl(x_{i+1}\sin\phi_i+y_{i+1}\cos\phi_i\bigr)B_i.
\end{aligned}
\label{eq:Q-transition}
\end{equation}
Hence $Q_{i+1}=Q_i$ if and only if
\begin{align}
-x_i&=x_{i+1}\cos\phi_i-y_{i+1}\sin\phi_i,\label{eq:Q-compat1}\\
y_i&=x_{i+1}\sin\phi_i+y_{i+1}\cos\phi_i.\label{eq:Q-compat2}
\end{align}
No division by $\sin\phi_i$ is used in these identities, so they remain valid for $\phi_i=\pi$.
\end{lemma}

\begin{proof}
Substitute the transition formulas \eqref{eq:torsion-convention}--\eqref{eq:N-transition} into $Q_{i+1}$.  The half-angle identities
\[
\cos\theta+x\sin\theta=1,\qquad
\sin\theta-x\cos\theta=x,
\qquad x=\tan\frac{\theta}{2},
\]
reduce its $T_i$-coefficient to $1$ and give \eqref{eq:Q-transition}.  Comparison with $Q_i=T_i-x_iN_i+y_iB_i$ proves the equivalence.
\end{proof}

\begin{proposition}[Conserved-vector formulation]\label{prop:conserved-vector}
Assume the standing turning-angle hypotheses and
\[
\phi_i\in(-\pi,\pi],\qquad 1\le i\le n-2.
\]
Set $x_i=\tan(\theta_i/2)$. Then the following are equivalent:
\begin{enumerate}[label=\textup{(\alph*)}]
\item the polygon is a discrete general helix with a non-orthogonal axis;
\item there exist real coefficients $y_i$, $1\le i\le n-1$, satisfying
\[
-x_i=x_{i+1}\cos\phi_i-y_{i+1}\sin\phi_i,\qquad
y_i=x_{i+1}\sin\phi_i+y_{i+1}\cos\phi_i;
\]
\item there exist real coefficients $y_i$, $1\le i\le n-1$, such that the Frenet-frame vector
\[
Q_i=T_i-x_iN_i+y_iB_i
\]
is independent of $i$.
\end{enumerate}
In this case $Q\cdot T_i=1$ for every edge tangent, including $T_0$, and the axis direction is the one-dimensional subspace $\mathbb R Q$; equivalently, the oriented helical direction is $Q/\lVert Q\rVert$. If at least one $\phi_i\ne\pi$, the compatible coefficients $y_i$ are unique; on the fully alternating locus $\phi_i=\pi$ for every $i$, they form a one-parameter family.
\end{proposition}

\begin{proof}
If the polygon has a non-orthogonal axis, the derivation of \eqref{eq:Qi} gives $U=cQ_i$ with fixed $U$ and $c\ne0$, so $Q_i$ is independent of $i$. Lemma~\ref{lem:Q-transition} shows that this is equivalent to the two compatibility relations, proving \textup{(a)}$\Rightarrow$\textup{(c)}$\Leftrightarrow$\textup{(b)}.

Conversely, if $Q_i=Q$ is constant, then $Q\cdot T_i=1$ for $1\le i\le n-1$. Using \eqref{eq:backward-tangent} at $i=1$ and $Q\cdot N_1=-x_1$ gives
\[
Q\cdot T_0=\cos\theta_1+x_1\sin\theta_1=1.
\]
Thus $Q\ne0$ and $U=Q/\|Q\|$ satisfies $T_i\cdot U=1/\|Q\|\ne0$ for all $i$, proving \textup{(c)}$\Rightarrow$\textup{(a)}.

If some $\phi_j\ne\pi$, then in the non-orthogonal situation $\phi_j\ne0$, and the compatibility equations at index $j$ determine the adjacent coefficients uniquely because $\sin\phi_j\ne0$; the remaining equations propagate this uniqueness through the chain. If all $\phi_i=\pi$, they reduce to $x_{i+1}=x_i$ and $y_{i+1}=-y_i$, leaving $y_1$ free.
\end{proof}

\begin{remark}[Discrete and smooth conserved vectors]
For comparison, in the smooth case one has
\[
\frac{d}{ds}\left(T+\frac{\kappa}{\tau}B\right)
=
\left(\frac{\kappa}{\tau}\right)'B
\]
on intervals where \(\kappa\tau\ne0\). Hence $T+(\kappa/\tau)B$ is constant when $\kappa/\tau$ is constant. This is the smooth analogue of the discrete vector $Q_i$.
\end{remark}

\begin{example}[Conserved vector for the symmetric example]\label{ex:conserved-benchmark}
For Example~\ref{ex:symmetric-visual-helix}, $U=(0,0,1)$ and $T_i\cdot U=3/5$. Hence the conserved vector is
\[
Q=\frac{U}{3/5}=\left(0,0,\frac53\right).
\]
At every interior Frenet frame the coefficients determined by
$Q=T_i-x_iN_i+y_iB_i$ satisfy the two compatibility equations of Proposition~\ref{prop:conserved-vector}. Thus this symmetric example satisfies the same conserved-vector equations used below. The proposition does not require constant edge lengths or rotational symmetry.
\end{example}

\begin{lemma}[Compatibility lemma]\label{lem:compatibility}
Let \(x_i>0\) for \(1\le i\le n-1\), and let
\(\phi_i\in(-\pi,\pi]\setminus\{0,\pi\}\) for \(1\le i\le n-2\).
Assume
\begin{equation}
\frac{x_{i+2}+x_{i+1}\cos\phi_{i+1}}{\sin\phi_{i+1}}
=
\frac{x_i+x_{i+1}\cos\phi_i}{\sin\phi_i},
\qquad 1\le i\le n-3.
\label{eq:finite-step}
\end{equation}
Define the endpoint values
\begin{align}
y_1&=\frac{x_2+x_1\cos\phi_1}{\sin\phi_1},
\label{eq:y-left}\\
y_{n-1}&=\frac{x_{n-2}+x_{n-1}\cos\phi_{n-2}}
{\sin\phi_{n-2}},
\label{eq:y-right}
\end{align}
and, for \(2\le j\le n-2\), define \(y_j\) by either of the equal
expressions
\begin{equation}
y_j=
\frac{x_{j-1}+x_j\cos\phi_{j-1}}{\sin\phi_{j-1}}
=
\frac{x_{j+1}+x_j\cos\phi_j}{\sin\phi_j}.
\label{eq:y-interior}
\end{equation}
The equality in \eqref{eq:y-interior} is precisely
\eqref{eq:finite-step} with index \(j-1\). Then, for every
\(1\le i\le n-2\),
\begin{align}
-x_i&=x_{i+1}\cos\phi_i-y_{i+1}\sin\phi_i,
\label{eq:compat1}\\
y_i&=x_{i+1}\sin\phi_i+y_{i+1}\cos\phi_i.
\label{eq:compat2}
\end{align}
\end{lemma}

\begin{proof}
For a fixed \(1\le i\le n-2\), the definition of \(y_{i+1}\) from the
left transition is
\[
y_{i+1}=
\frac{x_i+x_{i+1}\cos\phi_i}{\sin\phi_i}.
\]
Multiplication by \(\sin\phi_i\) immediately gives
\eqref{eq:compat1}.

The definition of \(y_i\) from the right transition is
\[
y_i=
\frac{x_{i+1}+x_i\cos\phi_i}{\sin\phi_i}.
\]
Using the preceding formula for \(y_{i+1}\),
\begin{align*}
x_{i+1}\sin\phi_i+y_{i+1}\cos\phi_i
&=
x_{i+1}\sin\phi_i+
\frac{\cos\phi_i(x_i+x_{i+1}\cos\phi_i)}{\sin\phi_i}\\
&=
\frac{x_{i+1}(\sin^2\phi_i+\cos^2\phi_i)
+x_i\cos\phi_i}{\sin\phi_i}\\
&=
\frac{x_{i+1}+x_i\cos\phi_i}{\sin\phi_i}
=y_i.
\end{align*}
This proves \eqref{eq:compat2}.  The endpoint formulas
\eqref{eq:y-left} and \eqref{eq:y-right} ensure that no nonexistent
indices \(0\) or \(n-1\) are invoked.
\end{proof}

\begin{example}\label{ex:compat-rotation}
For a single transition \(i\), the two formulas
\[
y_{i+1}=\frac{x_i+x_{i+1}\cos\phi_i}{\sin\phi_i},
\qquad
y_i=\frac{x_{i+1}+x_i\cos\phi_i}{\sin\phi_i}
\]
determine the two adjacent binormal coefficients.  The finite-step
relation is exactly the condition that the two definitions of an interior
coefficient \(y_j\) obtained from its left and right transitions coincide.  Figure~\ref{fig:compat-rotation} displays the underlying rotation.
\end{example}

\begin{figure}[!ht]
\centering
\resizebox{0.60\textwidth}{!}{\input{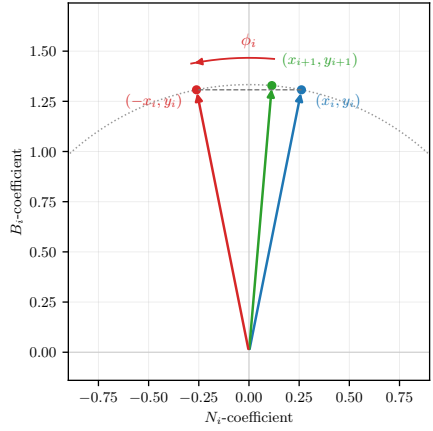}}
\caption{Rotation interpretation of the compatibility relations in Lemma~\ref{lem:compatibility}. In the $(N_i,B_i)$-coefficient plane, each transition is a rotation through $\phi_i$. Hence the radius $\sqrt{x_i^2+y_i^2}=\tan\alpha$ is conserved, while $x_i>0$ confines the phases to the right half-plane. The displayed data correspond to the first transition in Example~\ref{ex:nonconstant-helix}, with $\tan\alpha=4/3$ and $\phi_1=\arctan(25/86)$.}
\label{fig:compat-rotation}
\end{figure}

\section{Intrinsic characterizations and reconstruction}

\subsection{The generic local criterion}

We now eliminate the binormal coefficients from Proposition~\ref{prop:conserved-vector}. This can be done when the relevant torsion sines are nonzero.

\begin{theorem}[Finite-step characterization]\label{thm:finite-step}
Let \(P=(P_0,\ldots,P_n)\), \(n\ge3\), satisfy
\[
0<\theta_i<\pi,\qquad 1\le i\le n-1,
\]
and
\[
\phi_i\in(-\pi,\pi]\setminus\{0,\pi\},
\qquad 1\le i\le n-2.
\]
Set \(x_i=\tan(\theta_i/2)\). Then the following are equivalent:
\begin{enumerate}[label=\textup{(\roman*)}]
\item there exist a fixed unit vector \(U\) and a nonzero constant \(c\)
such that
\[
T_i\cdot U=c,\qquad 0\le i\le n-1;
\]
\item the Frenet data satisfy
\[
\boxed{
\frac{x_{i+2}+x_{i+1}\cos\phi_{i+1}}{\sin\phi_{i+1}}
=
\frac{x_i+x_{i+1}\cos\phi_i}{\sin\phi_i}
},
\qquad 1\le i\le n-3.
\]
\end{enumerate}
For \(n=3\), condition \textup{(ii)} is empty.
\end{theorem}

\begin{proof}
Assume first that \(T_i\cdot U=c\neq0\). By \eqref{eq:Qi},
\[
U=cQ_i,\qquad Q_i=T_i-x_iN_i+y_iB_i.
\]
Since \(U\) is fixed, \(Q_{i+1}=Q_i\). Lemma~\ref{lem:Q-transition} then gives, by comparison of the $N_i$- and $B_i$-components,
\begin{align}
-x_i&=x_{i+1}\cos\phi_i-y_{i+1}\sin\phi_i,\label{eq:necessity-C1}\\
y_i&=x_{i+1}\sin\phi_i+y_{i+1}\cos\phi_i.\label{eq:necessity-C2}
\end{align}
The first equation gives
\[
y_{i+1}
=
\frac{x_i+x_{i+1}\cos\phi_i}{\sin\phi_i}.
\]
For $1\le i\le n-3$, apply \eqref{eq:necessity-C1} and \eqref{eq:necessity-C2} at the next transition.  From \eqref{eq:necessity-C1} with index $i+1$,
\[
y_{i+2}=\frac{x_{i+1}+x_{i+2}\cos\phi_{i+1}}{\sin\phi_{i+1}}.
\]
Substitution into \eqref{eq:necessity-C2} with index $i+1$ gives
\begin{align*}
y_{i+1}
&=x_{i+2}\sin\phi_{i+1}+y_{i+2}\cos\phi_{i+1}\\
&=\frac{x_{i+2}+x_{i+1}\cos\phi_{i+1}}{\sin\phi_{i+1}}.
\end{align*}
Equating the two expressions for $y_{i+1}$ proves necessity.

Conversely, assume the finite-step relation. Define $y_1,\ldots,y_{n-1}$ by the endpoint and interior formulas in Lemma~\ref{lem:compatibility}. When $n=3$, there are only the two endpoint coefficients $y_1,y_2$ and no interior coefficient or scalar finite-step equation; the same lemma still gives the compatibility system. Thus, for every $n\ge3$,
\[
-x_i=x_{i+1}\cos\phi_i-y_{i+1}\sin\phi_i,
\qquad
y_i=x_{i+1}\sin\phi_i+y_{i+1}\cos\phi_i.
\]
Lemma~\ref{lem:Q-transition} therefore gives
\[
Q_{i+1}=Q_i,
\qquad
Q_i=T_i-x_iN_i+y_iB_i.
\]
Hence \(Q_i=Q\) for a fixed vector \(Q\), and orthonormality gives
\[
Q\cdot T_i=1,\qquad 1\le i\le n-1.
\]
It remains to recover the initial tangent $T_0$, because the Frenet frame is not defined at index $0$.  Using \eqref{eq:backward-tangent} with $i=1$ and $Q\cdot N_1=-x_1$,
\begin{align*}
Q\cdot T_0
&=\cos\theta_1\,(Q\cdot T_1)-\sin\theta_1\,(Q\cdot N_1)\\
&=\cos\theta_1+x_1\sin\theta_1=1,
\end{align*}
where the last equality is the half-angle identity already used above.  Thus $Q\cdot T_i=1$ for every $0\le i\le n-1$, so \(Q\neq0\). Define \(U=Q/\|Q\|\). Then
\[
T_i\cdot U=\frac1{\|Q\|},\qquad 0\le i\le n-1.
\]
Thus all edge tangents make a constant nonzero angle with the fixed direction \(U\).
\end{proof}

The theorem is the quotient form of the conserved-vector equations. In a non-orthogonal helix, $\phi_i=0$ is impossible, while $\phi_i=\pi$ is geometrically possible and gives the local backtracking relation $T_{i+1}=T_{i-1}$. Thus the quotient formula is used only on the generic branch. The results below also include $\phi_i=\pi$ and do not divide by $\sin\phi_i$.

\begin{remark}[Why simple cross multiplication does not remove the generic hypothesis]\label{rem:cross-multiplication}
Multiplying the quotient relation in Theorem~\ref{thm:finite-step} by the two sine factors gives a necessary identity that remains formally meaningful when a torsion angle equals $\pi$.  It is not, however, sufficient across such a step, because the sign relation between the adjacent binormal coefficients is then lost.  For example, take
\[
x_1=x_2=x_3=x_4=1,
\qquad
(\phi_1,\phi_2,\phi_3)=(0.7,\pi,1.2).
\]
The cross-multiplied identities hold for both admissible indices, and the $\phi_2=\pi$ step also has $x_2=x_3$.  Nevertheless the left transition gives
\[
y_2=\frac{1+\cos0.7}{\sin0.7},
\]
whereas the right transition gives
\[
y_3=\frac{1+\cos1.2}{\sin1.2}.
\]
The middle transition would require $y_3=-y_2$, which is impossible.  Thus no conserved vector exists for these data. Since any non-orthogonal helix would give the fixed vector $Q_i=U/c$ through equation~\eqref{eq:Qi}, these data are not helical. Hence this quotient criterion is restricted to the generic branch. Theorem~\ref{thm:phase-parametrization} treats the full non-orthogonal case.
\end{remark}

\begin{remark}[On the torsion-angle hypothesis]\label{rem:torsion-hypothesis}
With the complete branch choice $\phi_i\in(-\pi,\pi]$, the quotient formula requires $\phi_i\notin\{0,\pi\}$.  For a non-orthogonal discrete general helix, $\phi_i=0$ is impossible: the first compatibility equation would give $-x_i=x_{i+1}$, contradicting $x_i,x_{i+1}>0$.  In contrast, $\phi_i=\pi$ is possible.  In that case the undivided compatibility equations reduce to
\[
x_i=x_{i+1},\qquad y_i=-y_{i+1},
\]
so equivalently $\theta_i=\theta_{i+1}$.  Thus $\phi_i=\pi$ is excluded only by the quotient formula, not by the geometry.
\end{remark}

\begin{corollary}[Helical direction and angle reconstruction]\label{cor:axis-reconstruction}
Under the hypotheses of the theorem,
\[
Q=T_i-x_iN_i+y_iB_i
\]
is independent of \(i\), and
\[
\boxed{
U=
\frac{T_i-x_iN_i+y_iB_i}
{\sqrt{1+x_i^2+y_i^2}}
}.
\]
Furthermore, if \(\alpha\) denotes the common tangent-direction angle, then
\[
\boxed{
\cos\alpha=\frac1{\sqrt{1+x_i^2+y_i^2}}.
}
\]
The formula chooses the orientation for which $T_i\cdot U>0$; the corresponding axis direction is the unoriented one-dimensional subspace $\mathbb RQ$.
\end{corollary}

\begin{proof}
The compatibility equations can be written as
\[
\begin{pmatrix}-x_i\\ y_i\end{pmatrix}
=
\begin{pmatrix}
\cos\phi_i&-\sin\phi_i\\
\sin\phi_i&\cos\phi_i
\end{pmatrix}
\begin{pmatrix}x_{i+1}\\ y_{i+1}\end{pmatrix}.
\]
The matrix is orthogonal. Hence
\[
x_i^2+y_i^2=x_{i+1}^2+y_{i+1}^2.
\]
Therefore
\[
\|Q\|^2=1+x_i^2+y_i^2.
\]
Since \(Q\cdot T_i=1\), the stated formulas follow.
\end{proof}

\subsection{Complete non-orthogonal characterizations}

\begin{theorem}[Complete phase characterization]\label{thm:phase-parametrization}
Let $n\ge3$ and let
\[
0<\theta_i<\pi\quad(1\le i\le n-1),\qquad
\phi_i\in(-\pi,\pi]\quad(1\le i\le n-2)
\]
be discrete Frenet data. These data arise from a discrete general helix with non-orthogonal axis if and only if there exist
\[
\alpha\in(0,\pi/2),\qquad
\psi_1,\ldots,\psi_{n-1}\in(-\pi/2,\pi/2)
\]
such that
\begin{align}
\tan\frac{\theta_i}{2}&=\tan\alpha\,\cos\psi_i,
\qquad 1\le i\le n-1,\label{eq:phase-theta}\\
\phi_i&\equiv \pi-\psi_i-\psi_{i+1}\pmod{2\pi},
\qquad 1\le i\le n-2,\label{eq:phase-phi}
\end{align}
where $\phi_i$ is represented in $(-\pi,\pi]$. For each chosen oriented helical direction $U$ with $T_i\cdot U>0$, the corresponding parameters $\alpha$ and $\psi_i$ are uniquely determined by that direction. If at least one torsion angle is different from $\pi$, the Frenet data determine the non-orthogonal helical direction, and hence the phase parameters, uniquely. On the fully alternating locus $\phi_i=\pi$ for every $i$, helical data necessarily have a constant turning angle; for such data the same Frenet data admit a one-parameter family of helical directions (see Remark~\ref{rem:branch-overlap}).
\end{theorem}

\begin{proof}
Suppose first that the polygon is a non-orthogonal discrete general helix. Reverse the helical direction if necessary so that
\[
T_i\cdot U=\cos\alpha>0,
\qquad \alpha\in(0,\pi/2).
\]
Equation~\eqref{eq:Qi}, whose derivation uses only the constant-angle condition and $c\ne0$, gives
\[
U=(\cos\alpha)Q,
\qquad
Q=T_i-x_iN_i+y_iB_i.
\]
Since $U$ is unit, $\|Q\|=\sec\alpha$, and therefore $x_i^2+y_i^2=\tan^2\alpha$. Set $r=\tan\alpha$ and $w_i=x_i+\mathrm{i}y_i$. Since $x_i>0$, there is a unique $\psi_i\in(-\pi/2,\pi/2)$ such that
\[
w_i=r e^{\mathrm{i}\psi_i}.
\]
The compatibility equations, which follow directly from $Q_{i+1}=Q_i$ and remain valid when $\phi_i=\pi$, are equivalent to
\[
w_{i+1}=-e^{-\mathrm{i}\phi_i}\overline{w_i}.
\]
Comparing moduli and arguments gives \eqref{eq:phase-theta} and \eqref{eq:phase-phi}.

Conversely, choose $\alpha$ and the $\psi_i$ as in the statement and put
\[
r=\tan\alpha,
\qquad x_i=r\cos\psi_i,
\qquad y_i=r\sin\psi_i.
\]
Because $\cos\psi_i>0$, each $x_i>0$ and hence $0<\theta_i<\pi$. Equation \eqref{eq:phase-phi} is equivalent to
\[
w_{i+1}=-e^{-\mathrm{i}\phi_i}\overline{w_i},
\]
whose real and imaginary parts are
\[
-x_i=x_{i+1}\cos\phi_i-y_{i+1}\sin\phi_i,
\qquad
y_i=x_{i+1}\sin\phi_i+y_{i+1}\cos\phi_i.
\]
Lemma~\ref{lem:Q-transition} therefore gives $Q_{i+1}=Q_i$ for
\[
Q_i=T_i-x_iN_i+y_iB_i,
\]
without division by $\sin\phi_i$. As before, $Q\cdot T_i=1$ for all edge tangents, including $T_0$, and
\[
\|Q\|^2=1+r^2=\sec^2\alpha.
\]
Thus $U=Q/\|Q\|$ satisfies $T_i\cdot U=\cos\alpha$. For each choice of positive edge lengths, the discrete fundamental theorem realizes the prescribed Frenet data by a polygon unique up to an orientation-preserving Euclidean motion.

For a fixed oriented helical direction, the coefficients $y_i$ are fixed by $U=(\cos\alpha)Q$, so $r=\tan\alpha=\sqrt{x_i^2+y_i^2}$ and, since $x_i>0$, each $\psi_i=\arg(x_i+\mathrm{i}y_i)$ has a unique representative in $(-\pi/2,\pi/2)$. If some $\phi_j\ne\pi$, then $\sin\phi_j\ne0$ because $\phi_j=0$ is impossible in the present non-orthogonal setting; the compatibility equations at that step determine the adjacent $y$-coefficients, and the remaining equations propagate them uniquely through the chain. Hence the helical direction is unique. If instead $\phi_i=\pi$ for every $i$, then $x_{i+1}=x_i$ and $y_{i+1}=-y_i$ for all $i$, so the tangents alternate between two directions. The value of $y_1$ is free, producing the one-parameter family described explicitly in Example~\ref{ex:nonunique-axis}.
\end{proof}

The phase equations can also be written as the following linear condition.

\begin{theorem}[Linear-subspace characterization]\label{thm:linear-subspace}
Let the Frenet data satisfy the hypotheses of Theorem~\ref{thm:phase-parametrization}, and set
\[
x_i=\tan\frac{\theta_i}{2},\qquad x=(x_1,\ldots,x_{n-1})^{\mathsf T}.
\]
Define torsion-dependent offsets by
\[
\beta_1=0,\qquad \beta_{i+1}=\pi-\phi_i-\beta_i,
\quad 1\le i\le n-2,
\]
and vectors $a,b\in\mathbb R^{n-1}$ by
\[
a_i=\cos\beta_i,
\qquad
b_i=-(-1)^{i-1}\sin\beta_i.
\]
Then the data arise from a discrete general helix with a non-orthogonal axis if and only if
\[
\boxed{\ x\in\operatorname{span}\{a,b\}.\ }
\]
Equivalently,
\[
\operatorname{rank}[a\ b\ x]=\operatorname{rank}[a\ b].
\]
If $\operatorname{rank}[a\ b]=2$, the coefficients $u,v$ in $x=ua+vb$ are unique, $u=x_1>0$, and
\[
\tan\alpha=\sqrt{u^2+v^2},
\qquad
\psi_1=\operatorname{atan2}(v,u).
\]
If $\operatorname{rank}[a\ b]=1$ and the membership condition holds, then necessarily $\phi_i=\pi$ for every $i$ and $x_i$ is constant; this is the fully alternating nonuniqueness case.
\end{theorem}

\begin{proof}
Assume first that the data are helical and use the phase variables of Theorem~\ref{thm:phase-parametrization}. The recursion for $\beta_i$ and the congruence \eqref{eq:phase-phi} imply inductively
\[
\psi_i\equiv(-1)^{i-1}\psi_1+\beta_i\pmod{2\pi}.
\]
With $r=\tan\alpha$, put
\[
u=r\cos\psi_1,\qquad v=r\sin\psi_1.
\]
Then
\[
\begin{aligned}
x_i
&=r\cos\bigl((-1)^{i-1}\psi_1+\beta_i\bigr)\\
&=u\cos\beta_i-v(-1)^{i-1}\sin\beta_i
=ua_i+vb_i,
\end{aligned}
\]
so $x\in\operatorname{span}\{a,b\}$. Since $\beta_1=0$, one has $a_1=1$, $b_1=0$, and therefore $u=x_1>0$.

Conversely, suppose $x=ua+vb$. Again $u=x_1>0$. Set
\[
r=\sqrt{u^2+v^2},\qquad \psi_1=\operatorname{atan2}(v,u)\in(-\pi/2,\pi/2),
\]
and let $\psi_i$ be the unique representative in $(-\pi/2,\pi/2)$ of
$(-1)^{i-1}\psi_1+\beta_i$ modulo $2\pi$. Such a representative exists because
\[
x_i=r\cos\bigl((-1)^{i-1}\psi_1+\beta_i\bigr)>0.
\]
The defining recursion gives
\[
\beta_i+\beta_{i+1}=\pi-\phi_i,
\]
so the alternating $\psi_1$ terms cancel and
\[
\phi_i\equiv\pi-\psi_i-\psi_{i+1}\pmod{2\pi}.
\]
Also $x_i=r\cos\psi_i$. Theorem~\ref{thm:phase-parametrization} therefore gives a non-orthogonal discrete general helix with $\alpha=\arctan r$.

Finally, if $\operatorname{rank}[a\ b]=1$, then $b=0$ because $(a_1,b_1)=(1,0)$. Hence $\sin\beta_i=0$ and therefore $a_i=\cos\beta_i\in\{\pm1\}$ for every $i$. The membership relation becomes $x_i=ua_i$; since $u=x_1>0$ and every $x_i>0$, it follows that $a_i=1$ for all $i$. Hence each $\beta_i$ is an even multiple of $\pi$, and the recursion with $\phi_i\in(-\pi,\pi]$ forces $\phi_i=\pi$ for every $i$. Thus $x_i=u$ is constant. Conversely this fully alternating case has $a=(1,\ldots,1)$ and $b=0$.
\end{proof}

\begin{remark}[A direct computational test]\label{rem:linear-test}
Theorem~\ref{thm:linear-subspace} also gives a simple test on the full non-orthogonal branch, including $\phi_i=\pi$. Compute $\beta$, form $a,b$, and check whether $x$ belongs to their span. For numerical data one may solve $\min_{u,v}\|x-ua-vb\|$. For exact data, the residual is zero exactly in the helical case. The rank equality in the theorem is essential on the degenerate branch: the weaker condition $\operatorname{rank}[a\ b\ x]\le2$ would become vacuous when $b=0$.

The rank drop itself has a simple torsion-angle description. Since $a_1=1$ and $b_1=0$,
\[
\operatorname{rank}[a\ b]=1
\quad\Longleftrightarrow\quad
b=0
\quad\Longleftrightarrow\quad
\phi_i\in\{0,\pi\}\quad\text{for every }i.
\]
Indeed $b=0$ is equivalent to $\sin\beta_i=0$ for every $i$, and the recursion for $\beta_i$ then gives precisely these two torsion values. If the membership condition also holds with $x_i>0$, the proof of Theorem~\ref{thm:linear-subspace} rules out every zero-torsion transition and leaves only the fully alternating case $\phi_i=\pi$ for all $i$. More generally, a single zero-torsion step is automatically rejected, independently of the rank. If $\phi_j=0$, then $\beta_{j+1}=\pi-\beta_j$, and hence
\[
(a_{j+1},b_{j+1})=-(a_j,b_j).
\]
Therefore every representation $x=ua+vb$ would satisfy $x_{j+1}=-x_j$, contradicting $x_j,x_{j+1}>0$. Thus the linear-subspace test excludes zero torsion directly from the positive half-angle data. Consequently, on the non-orthogonal helical locus the condition in Theorem~\ref{thm:phase-parametrization} that at least one torsion angle differs from $\pi$ is equivalent to $\operatorname{rank}[a\ b]=2$ in Theorem~\ref{thm:linear-subspace}; the apparent difference between the two uniqueness hypotheses is therefore only notational on that locus.
\end{remark}

\begin{figure}[!ht]
\centering
\resizebox{0.97\textwidth}{!}{\input{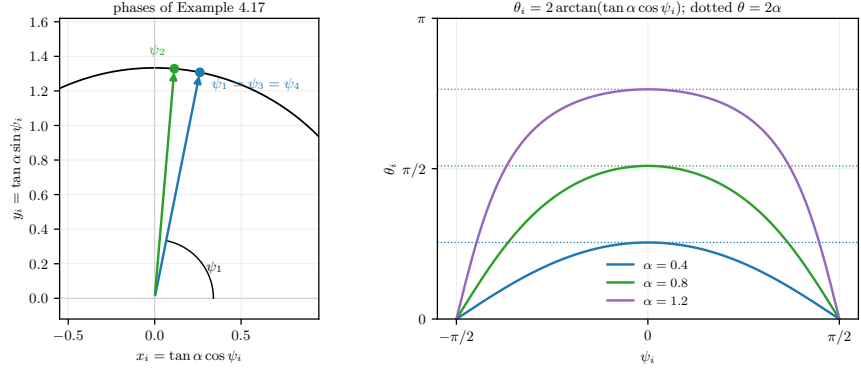}}
\caption{The phase characterization of Theorem~\ref{thm:phase-parametrization}. Left: the coefficients $w_i=x_i+\mathrm{i}y_i$ lie on the circle of radius $\tan\alpha$, and $x_i>0$ restricts the phases $\psi_i$ to $(-\pi/2,\pi/2)$; the four phases shown are those of Example~\ref{ex:nonconstant-helix}, where $\psi_3=\psi_1$ forces $\phi_2=\phi_1$. Right: the turning angle as a function of the phase, $\theta_i=2\arctan(\tan\alpha\cos\psi_i)$, for three values of $\alpha$; the dotted lines $\theta=2\alpha$ display the sharp bound of Corollary~\ref{cor:turning-bound}.}
\label{fig:phase-parametrisation}
\end{figure}

\begin{corollary}[Turning-angle range and sharpness]\label{cor:turning-bound}
For a non-orthogonal discrete general helix with acute helix angle $\alpha$,
\[
\boxed{\theta_i\le 2\alpha\qquad(1\le i\le n-1).}
\]
Equality holds exactly when $y_i=0$, equivalently $\psi_i=0$.
Conversely, given arbitrary turning angles $\theta_1,\ldots,\theta_{n-1}\in(0,\pi)$, every
\[
\alpha\in\left[\frac12\max_i\theta_i,\frac\pi2\right)
\]
occurs as the acute helix angle of some non-orthogonal discrete general helix with those turning angles and suitable torsion angles. Thus the turning-angle sequence alone imposes no further compatibility condition; the right panel of Figure~\ref{fig:phase-parametrisation} displays the bound and its equality case. If the lower endpoint $\alpha=\frac12\max_i\theta_i$ is chosen and all turning angles are equal, then all $\psi_i=0$, hence all $\phi_i=\pi$; this is exactly the alternating planar overlap described in Remark~\ref{rem:branch-overlap}.
\end{corollary}

\begin{proof}
The necessity follows from \eqref{eq:phase-theta}:
\[
\tan\frac{\theta_i}{2}=\tan\alpha\cos\psi_i\le\tan\alpha.
\]
Monotonicity of $\tan$ on $(0,\pi/2)$ gives $\theta_i/2\le\alpha$, with equality exactly when $\psi_i=0$, equivalently $y_i=0$.

Conversely, fix $\alpha\ge\frac12\max_i\theta_i$ with $\alpha<\pi/2$. Then
\[
0<\frac{\tan(\theta_i/2)}{\tan\alpha}\le1.
\]
Choose independently
\[
\psi_i=\pm\arccos\!\left(\frac{\tan(\theta_i/2)}{\tan\alpha}\right)
\in(-\pi/2,\pi/2),
\]
where the two signs coincide when the arccosine is zero, and define $\phi_i$ by \eqref{eq:phase-phi} with its representative in $(-\pi,\pi]$. Theorem~\ref{thm:phase-parametrization} then produces the required non-orthogonal discrete general helix. Notice that $\phi_i=0$ cannot occur because $\psi_i+\psi_{i+1}\in(-\pi,\pi)$.
\end{proof}

\begin{proposition}[Realizability of prescribed torsion-angle data]\label{prop:torsion-realizability}
Fix torsion angles
\[
\phi_i\in(-\pi,\pi],\qquad 1\le i\le n-2,
\]
and define
\[
\beta_1=0,\qquad \beta_{i+1}=\pi-\phi_i-\beta_i.
\]
There exists a non-orthogonal discrete general helix having these torsion angles and some turning angles in $(0,\pi)$ if and only if
\[
\bigcap_{i=1}^{n-1}
\left\{
\psi\in\mathbb R/2\pi\mathbb Z:
\cos\bigl((-1)^{i-1}\psi+\beta_i\bigr)>0
\right\}
\ne\varnothing.
\]
Thus arbitrary turning-angle data can be realized after a suitable choice of torsion angles, as in Corollary~\ref{cor:turning-bound}, but prescribed torsion-angle data need not be realizable.
\end{proposition}

\begin{proof}
By Theorem~\ref{thm:phase-parametrization} and the recursion for $\beta_i$, every non-orthogonal helical phase sequence has the form
\[
\psi_i\equiv(-1)^{i-1}\psi_1+\beta_i\pmod{2\pi}.
\]
The condition $\psi_i\in(-\pi/2,\pi/2)$ is equivalent to $\cos\psi_i>0$, so necessity follows.

Conversely, choose $\psi_1$ in the displayed intersection and any $r>0$. Put
\[
\psi_i\equiv(-1)^{i-1}\psi_1+\beta_i\pmod{2\pi},
\qquad
x_i=r\cos\psi_i>0,
\]
with representatives $\psi_i\in(-\pi/2,\pi/2)$, and define
\[
\theta_i=2\arctan x_i.
\]
The recursion for $\beta_i$ gives
\[
\phi_i\equiv\pi-\psi_i-\psi_{i+1}\pmod{2\pi}.
\]
Theorem~\ref{thm:phase-parametrization} then yields the required non-orthogonal discrete general helix.
\end{proof}

\begin{example}[A torsion-angle sequence that is not realizable]\label{ex:torsion-nonrealizable}
Take $n=4$ and
\[
(\phi_1,\phi_2)=(0.1,-0.1).
\]
Then
\[
\beta=(0,\pi-0.1,0.2).
\]
The first and second positivity conditions force, modulo $2\pi$,
\[
\psi_1\in\left(\frac\pi2-0.1,\frac\pi2\right),
\]
whereas the third requires
\[
\psi_1<\frac\pi2-0.2.
\]
These conditions are incompatible. Hence no non-orthogonal discrete general helix can have this torsion-angle sequence, irrespective of the turning angles.
\end{example}

\begin{remark}[Generic codimension]\label{rem:codimension}
The linear criterion also gives a short rigorous derivation of the generic codimension. On the branch $\phi_i\notin\{0,\pi\}$, one already has $\sin\beta_2=\sin\phi_1\ne0$. Since $(a_1,b_1)=(1,0)$, this implies that the two columns $a$ and $b$ of Theorem~\ref{thm:linear-subspace} are linearly independent.

Consider the open parameter set
\[
\begin{aligned}
\mathcal D=\bigl\{(\phi,u,v)\in{}&
\bigl(( -\pi,0)\cup(0,\pi)\bigr)^{n-2}\times\mathbb R^2:\\
&u>0,\quad x_i=ua_i(\phi)+vb_i(\phi)>0
\ \text{for all }i\bigr\}.
\end{aligned}
\]
and the map
\[
\Xi:\mathcal D\longrightarrow\mathbb R^{2n-3},\qquad
(\phi,u,v)\longmapsto
\bigl(2\arctan x_1,\ldots,2\arctan x_{n-1};\phi_1,\ldots,\phi_{n-2}\bigr).
\]
The map is smooth. If a tangent vector belongs to the kernel of $d\Xi$, its torsion components vanish because the last $n-2$ output coordinates are the identity. With $\dot\phi=0$, the turning-angle components give
\[
0=\frac{2}{1+x_i^2}(\dot u\,a_i+\dot v\,b_i),
\qquad 1\le i\le n-1.
\]
Linear independence of $a,b$ yields $\dot u=\dot v=0$, so $d\Xi$ has rank $n$. Moreover $\Xi$ is injective on $\mathcal D$: the torsion coordinates recover $\phi$, while $u=x_1$ and, since $b_2\ne0$, $v=(x_2-u a_2)/b_2$. The inverse on the image is therefore smooth. Hence $\Xi$ is an embedding onto its image, and the generic helical locus is locally an $n$-dimensional embedded submanifold of the $(2n-3)$-dimensional Frenet-angle space. Its codimension is
\[
(2n-3)-n=n-3,
\]
exactly the number of scalar quotient equations in Theorem~\ref{thm:finite-step}. The rank-degenerate helical locus is the fully alternating two-tangent overlap described in Remark~\ref{rem:branch-overlap}.
\end{remark}

\begin{remark}[Backtracking torsion angles]\label{rem:pi-branch}
In the setting of Theorem~\ref{thm:phase-parametrization}, the complete phase characterization includes the case excluded by the quotient formula:
\[
\phi_i=\pi
\quad\Longleftrightarrow\quad
\psi_{i+1}=-\psi_i
\quad\Longleftrightarrow\quad
T_{i+1}=T_{i-1}.
\]
No singular quotient is involved. In particular, a nonplanar discrete general helix may contain such a backtracking step. If every torsion angle equals $\pi$, one reaches the alternating two-tangent locus of Remark~\ref{rem:branch-overlap}, where helical-direction nonuniqueness and phase-map degeneracy occur simultaneously.
\end{remark}

\begin{example}[A nonplanar helix with $\phi_1=\pi$]\label{ex:pi-helix}
Take
\[
\alpha=0.9,
\qquad
(\psi_1,\psi_2,\psi_3,\psi_4)=(0.6,-0.6,0.4,1.0).
\]
Theorem~\ref{thm:phase-parametrization} gives
\[
\phi_1=\pi,
\qquad
\phi_2=-\pi+0.2,
\qquad
\phi_3=\pi-1.4,
\]
and
\[
\theta_i=2\arctan(\tan(0.9)\cos\psi_i)\in(0,\pi).
\]
Thus the first torsion step is the backtracking case $T_2=T_0$, while the subsequent torsion angles are neither $0$ nor $\pi$. Hence the reconstructed tangent sequence is not contained in a plane, yet every tangent satisfies $T_i\cdot U=\cos(0.9)$ for the helical direction produced in the proof. This example shows why the case $\phi_i=\pi$ must be kept; see Figure~\ref{fig:pi-helix}.
\end{example}

\begin{figure}[!ht]
\centering
\resizebox{0.97\textwidth}{!}{\input{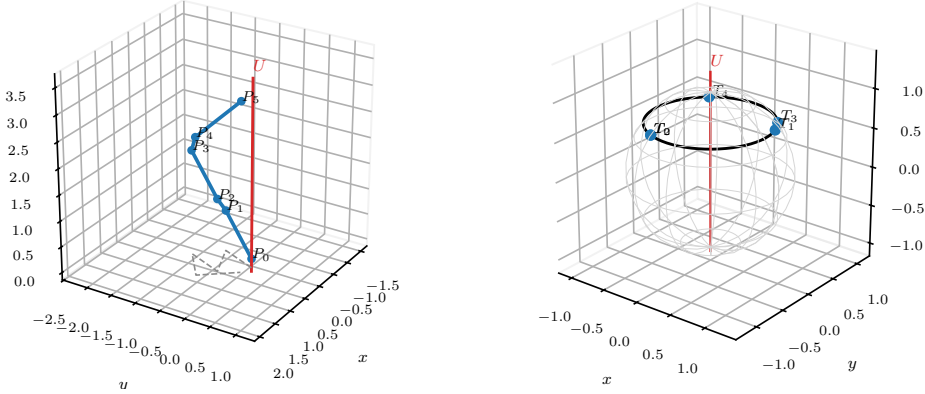}}
\caption{The nonplanar general helix of Example~\ref{ex:pi-helix}, reconstructed with unit edge lengths and rotated so that the helical direction is vertical. Left: a perspective view together with its horizontal floor projection, chosen to make the change of azimuth and the nonplanarity visible. Right: the tangent indicatrix on the small circle $X\cdot U=\cos\alpha$ with $\alpha=0.9$; the backtracking step $\phi_1=\pi$ gives $T_2=T_0$. This is a general helix in the constant-tangent-angle sense and need not resemble a circular screw helix.}
\label{fig:pi-helix}
\end{figure}

\begin{example}[Nonuniqueness on the fully alternating locus]\label{ex:nonunique-axis}
Let $n=3$ and prescribe
\[
\theta_1=\theta_2=0.8,\qquad \phi_1=\pi.
\]
The prescribed data have $\theta_1=\theta_2$ and $\phi_1=\pi$; substituting these values into the tangent transition formula~\eqref{eq:T-transition} gives $T_2=T_0$. Hence the tangent set consists of the two distinct vectors $T_0$ and $T_1$. A unit vector $U$ is an admissible helical direction exactly when
\[
(T_1-T_0)\cdot U=0.
\]
Hence the admissible unit helical directions form the great circle $S^2\cap(T_1-T_0)^\perp$.  Except for the two plane normals, these directions have $T_i\cdot U\ne0$, and the corresponding acute helix angle varies continuously. Thus the same Frenet data admit a one-parameter family of non-orthogonal helical directions. This explains the nonuniqueness in Theorem~\ref{thm:phase-parametrization}; see Figure~\ref{fig:axis-nonuniqueness}.
\end{example}

\begin{figure}[!ht]
\centering
\resizebox{0.97\textwidth}{!}{\input{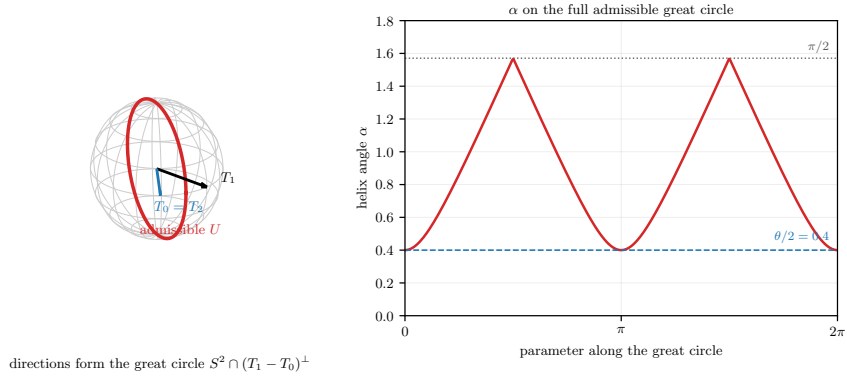}}
\caption{Nonuniqueness of the helical direction in Example~\ref{ex:nonunique-axis}. Left: only the two directions $T_0=T_2$ and $T_1$ occur, and every unit vector on the great circle $S^2\cap(T_1-T_0)^{\perp}$ is an admissible helical direction. Right: the acute angle between the tangents and the directions on that great circle. Over the full admissible great circle it ranges in $[\theta/2,\pi/2]$. The minimum $\theta/2$ is the equality case $\theta_i=2\alpha$ of Corollary~\ref{cor:turning-bound}; the two points with $\alpha=\pi/2$ are the plane-normal directions and belong to the orthogonal branch. Restricting to non-orthogonal helical directions therefore gives $\alpha\in[\theta/2,\pi/2)$.}
\label{fig:axis-nonuniqueness}
\end{figure}

\begin{remark}[Edge-length independence]\label{rem:length-independence}
The helical conditions involve only the unit edge tangents, or equivalently the Frenet-angle data. Therefore positive edge lengths may be changed independently without changing discrete general helicity. In particular, the unit edge lengths used in Example~\ref{ex:nonconstant-helix} are only a convenient choice. This statement concerns the exact polygonal condition. The estimates in Section~5 instead assume uniform arclength sampling $s_i=s_0+ih$ of a fixed smooth curve. Arbitrary independent changes of the edge lengths generally destroy that sampling model, so the stated $O(h^2)$ and $O(h)$ consistency orders are not asserted for such nonuniform resamplings.  Figure~\ref{fig:edge-length} illustrates the exact statement.
\end{remark}

\begin{figure}[!ht]
\centering
\resizebox{0.97\textwidth}{!}{\input{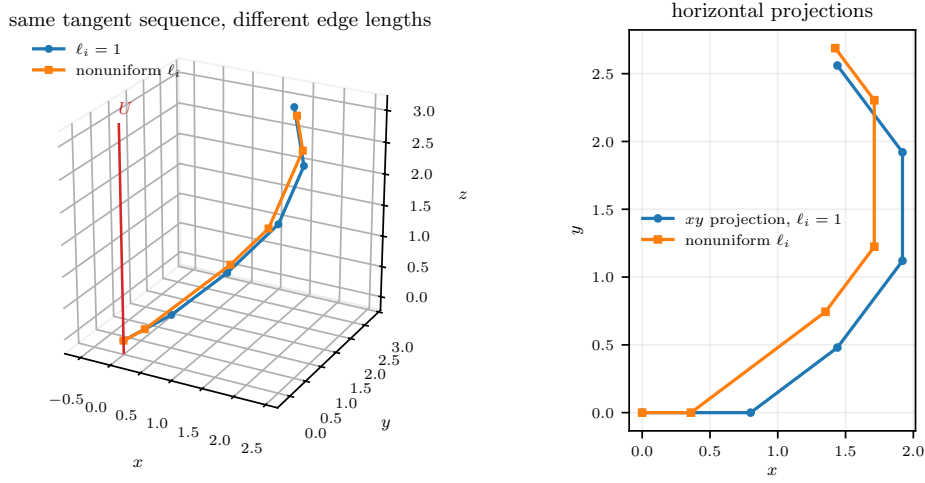}}
\caption{Edge-length independence, Remark~\ref{rem:length-independence}. Left: perspective views of two polygons built from the same tangents of Example~\ref{ex:nonconstant-helix}, one with unit edges and one with edge lengths $(0.45,1.55,0.75,1.35,0.60)$. Right: their horizontal projections. The edge lengths change the realized polygon but not the unit tangent sequence, Frenet-angle data, helical direction $U$, or helix angle.}
\label{fig:edge-length}
\end{figure}

\begin{remark}[No closed non-orthogonal helices]\label{rem:no-closed}
A polygonal discrete general helix with $T_i\cdot U=c\ne0$ and positive edge lengths cannot close. After orienting $U$ so that $c>0$,
\[
\left(\sum_i \ell_iT_i\right)\cdot U=c\sum_i\ell_i>0,
\]
so the total edge displacement cannot vanish. Closed examples are possible only in the orthogonal-axis branch $c=0$.
\end{remark}

\begin{example}[A nonconstant discrete general helix]\label{ex:nonconstant-helix}
Consider
\[
\begin{aligned}
T_0&=\left(\frac45,0,\frac35\right),&
T_1&=\left(\frac{16}{25},\frac{12}{25},\frac35\right),\\
T_2&=\left(\frac{12}{25},\frac{16}{25},\frac35\right),&
T_3&=\left(0,\frac45,\frac35\right),\\
T_4&=\left(-\frac{12}{25},\frac{16}{25},\frac35\right).
\end{aligned}
\]
Every \(T_i\) is a unit vector and
\[
T_i\cdot(0,0,1)=\frac35.
\]
Thus the fixed helical direction is \(U=(0,0,1)\) and
\[
\cos\alpha=\frac35.
\]
The resulting polygon and a representative line parallel to its axis direction are shown in Figure~\ref{fig:nonconstant-helix}.

Taking unit edge lengths and \(P_0=(0,0,0)\) gives
\[
\begin{aligned}
P_1&=\left(\frac45,0,\frac35\right),\\
P_2&=\left(\frac{36}{25},\frac{12}{25},\frac65\right),\\
P_3&=\left(\frac{48}{25},\frac{28}{25},\frac95\right),\\
P_4&=\left(\frac{48}{25},\frac{48}{25},\frac{12}{5}\right),\\
P_5&=\left(\frac{36}{25},\frac{64}{25},3\right).
\end{aligned}
\]
The turning data are nonconstant; for instance,
\[
T_0\cdot T_1=\frac{109}{125},
\qquad
T_1\cdot T_2=\frac{609}{625}.
\]
Hence \(\theta_1\neq\theta_2\).  More precisely,
\[
x_1=\frac{2\sqrt{26}}{39},\qquad
x_2=\frac{2\sqrt{1234}}{617},\qquad
x_3=x_4=\frac{2\sqrt{26}}{39}.
\]
Using the signed-angle convention \eqref{eq:torsion-convention} and evaluating the angle with $\operatorname{atan2}$ gives
\[
\phi_1=\phi_2=\arctan\frac{25}{86},
\qquad
\phi_3=\arctan\frac{5}{12}.
\]
The equality $\phi_1=\phi_2$ can be read directly from Theorem~\ref{thm:phase-parametrization}. Here $\tan\alpha=4/3$ and the corresponding phase variables satisfy $\psi_3=\psi_1$; hence
\[
\phi_2\equiv\pi-\psi_2-\psi_3=\pi-\psi_1-\psi_2\equiv\phi_1\pmod{2\pi}.
\]
Thus the equality is a consequence of the phase symmetry of this example rather than of a constant-torsion assumption.  The torsion data are nevertheless nonconstant because $\phi_3\ne\phi_2$. Numerically,
\[
(\phi_1,\phi_2,\phi_3)
\approx(0.282900854,0.282900854,0.394791120).
\]
The corresponding binormal coefficients are available in closed form:
\[
y_1=y_3=y_4=\frac{10\sqrt{26}}{39},
\qquad
y_2=\sqrt{\frac{9800}{5553}}.
\]
The two available finite-step identities therefore hold exactly. Their common values are approximately
$1.328462198722335$ and $1.307440900921227$, respectively.
Finally, direct substitution of these exact coefficients into the Frenet-frame expression gives, at every available interior frame,
\[
Q_i=T_i-x_iN_i+y_iB_i=(0,0,5/3).
\] Hence
\[
U=(0,0,1),\qquad \cos\alpha=3/5,
\]
exactly as prescribed by the original tangent data.  In addition, all
five points of the tangent indicatrix lie on the small circle
\[
S^2\cap\{z=3/5\}.
\]
Thus this example verifies the fixed-direction condition, the small-circle condition, the conserved vector $Q=(0,0,5/3)$, and the finite-step relation. Both the turning data and the torsion data are nonconstant.
\end{example}

\begin{figure}[!ht]
\centering
\resizebox{0.97\textwidth}{!}{\input{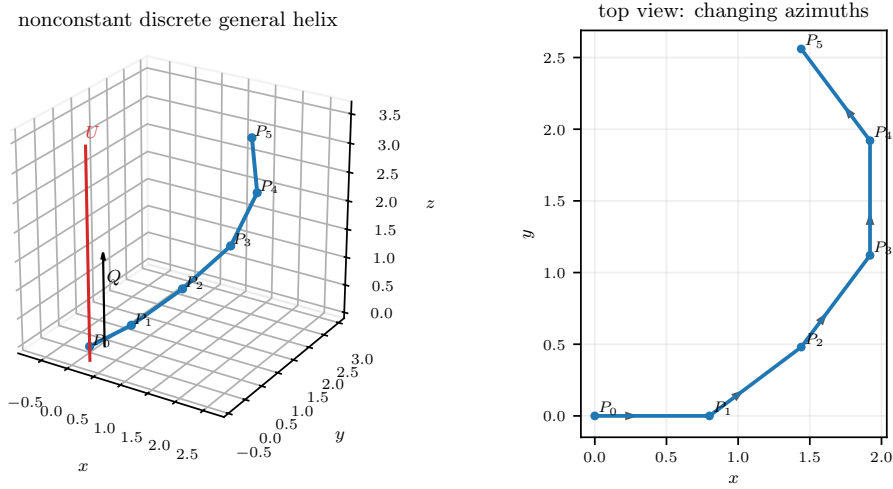}}
\caption{The nonconstant discrete general helix of Example~\ref{ex:nonconstant-helix} with unit edge lengths. Left: a perspective view of the polygon, a representative line parallel to the fixed direction $U=(0,0,1)$, and the conserved vector $Q=(0,0,5/3)$ drawn with a small lateral offset for readability. Right: the horizontal projection, which makes the changing edge azimuths visible. Here $U$ is the helical direction and $\mathbb R U$ is the axis direction; a general helix need not wind around a distinguished affine line as a circular helix does.}
\label{fig:nonconstant-helix}
\end{figure}

\section{Planar branch and smooth consistency}

\begin{proposition}[Orthogonal-axis case]\label{prop:orthogonal-axis}
Independently of whether the Frenet angles of Section~2 are defined at every vertex, a polygon satisfies
\[
T_i\cdot U=0
\]
for a fixed nonzero vector \(U\) and every edge \(i\) if and only if all
of its vertices lie in an affine plane orthogonal to \(U\).
\end{proposition}

\begin{proof}
If \(T_i\cdot U=0\), every edge vector \(P_{i+1}-P_i\) lies in \(U^\perp\).
Hence \(P_i-P_0\in U^\perp\) for every \(i\), and all vertices lie in
\(P_0+U^\perp\). Conversely, if all vertices lie in an affine plane with
normal \(U\), every edge tangent is orthogonal to \(U\).
\end{proof}

\begin{example}\label{ex:planar}
Any nondegenerate planar polygon in the \(xy\)-plane satisfies
\[
T_i\cdot(0,0,1)=0.
\]
This is the orthogonal branch. It is outside the generic hypotheses of Theorem~\ref{thm:finite-step}: a convex planar polygon has $\phi_i=0$, while a planar polygon with reversals of the oriented binormal may have $\phi_i\in\{0,\pi\}$; see Figure~\ref{fig:planar-case}.
\end{example}

\begin{figure}[!ht]
\centering
\resizebox{0.97\textwidth}{!}{\input{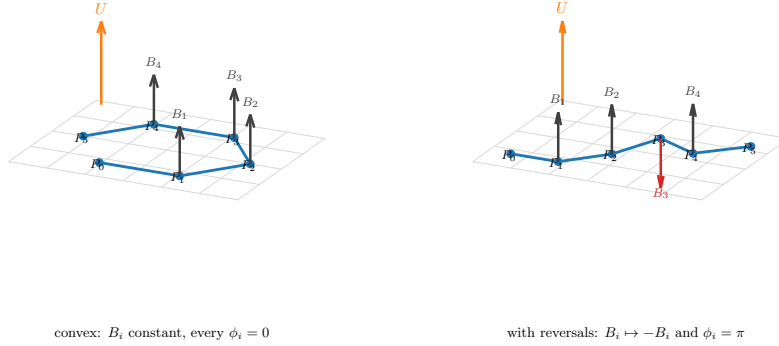}}
\caption{Planar polygons in the orthogonal-axis branch of Example~\ref{ex:planar}, drawn inside their plane in space together with the orthogonal helical direction $U$. Left: a convex polygon, where the oriented binormal $B_i$ is constant and every $\phi_i=0$. Right: a polygon whose turning direction reverses; at each reversal $B_i$ is replaced by its opposite and $\phi_i=\pi$. Both cases lie outside the generic hypotheses of Theorem~\ref{thm:finite-step}.}
\label{fig:planar-case}
\end{figure}

\begin{corollary}[Non-orthogonal/planar dichotomy]\label{cor:dichotomy}
Every discrete general helix in the sense of Definition~\ref{def:discrete-general-helix} belongs, for any chosen helical direction, to one of two branches. If $T_i\cdot U\ne0$, its Frenet data are described by Theorem~\ref{thm:phase-parametrization} whenever the standing turning-angle assumptions hold. If $T_i\cdot U=0$, the polygon is planar by Proposition~\ref{prop:orthogonal-axis}.
\end{corollary}

\begin{proof}
This is immediate from the alternatives $T_i\cdot U=c\ne0$ and $c=0$, together with Theorem~\ref{thm:phase-parametrization} and Proposition~\ref{prop:orthogonal-axis}.
\end{proof}

\begin{remark}[The unique overlap of the two axis branches]\label{rem:branch-overlap}
The two branches in Corollary~\ref{cor:dichotomy} are alternatives for a chosen helical direction, not necessarily mutually exclusive properties of the polygon.  A polygon can admit both an orthogonal axis and a non-orthogonal axis precisely when its tangent set consists of two distinct directions. Indeed, if the polygon is planar and also satisfies $T_i\cdot U=c\ne0$, its tangents lie in the intersection of a great circle with a noncentral plane section of $S^2$, which contains at most two points. Since consecutive tangents are distinct, they must alternate. Equivalently,
\[
\phi_i=\pi\quad\text{for every }i,
\]
with a constant turning angle. Conversely, such an alternating two-direction polygon is planar and, as Example~\ref{ex:nonunique-axis} shows, has a one-parameter family of non-orthogonal helical directions as well as the orthogonal plane-normal direction. Thus the overlap locus is exactly the nonuniqueness locus of Theorem~\ref{thm:phase-parametrization}; it is also the rank-degenerate locus excluded from the generic codimension statement in Remark~\ref{rem:codimension}.  The same geometry is visualized in Figure~\ref{fig:axis-nonuniqueness}.
\end{remark}

\begin{proposition}[Second-order smooth consistency]\label{prop:smooth-consistency}
Let \(\gamma\in C^6(I,\mathbb R^3)\) be a unit-speed space curve.
Assume that, on a compact subinterval \(J\Subset I\),
\[
\kappa(s)\ge \kappa_0>0,\qquad |\tau(s)|\ge\tau_0>0.
\]
Sample it uniformly by
\[
P_i=\gamma(s_i),\qquad s_i=s_0+ih.
\]
For the quotient indexed by $i$, call $C_i=[s_{i-1},s_{i+2}]$ its associated four-point cell. All quotient estimates below are uniform whenever $C_i\subset J$; estimates for the defect $D_i^h$, which involves two consecutive quotients, are uniform whenever $C_i\cup C_{i+1}=[s_{i-1},s_{i+3}]\subset J$.
Let \(\theta_i^h\) and \(\phi_i^h\) be the turning and signed torsion
angles of the chord polygon and set
\[
x_i^h=\tan\frac{\theta_i^h}{2}.
\]
There exists $h_0=h_0(\kappa_0,\tau_0,J)>0$ such that, for $0<h<h_0$ and all cells under consideration,
\[
0<\theta_i^h<\pi,\qquad \phi_i^h\notin\{0,\pi\}.
\]
Hence the sampled polygon lies in the generic Frenet branch and all quotients below are well defined.
With $m_i=s_i+h/2$, the midpoint convention compatible with the chord tangents gives the
asymptotic expansions
\[
x_i^h=\frac h2\kappa(s_i)+O(h^3),
\]
and
\[
\phi_i^h=h\tau\left(s_i+\frac h2\right)+O(h^3).
\]
Consequently,
\[
\frac{x_i^h+x_{i+1}^h\cos\phi_i^h}{\sin\phi_i^h}
=
\frac{\kappa}{\tau}\left(s_i+\frac h2\right)+O(h^2).
\]
If
\[
D_i^h:=
\frac{x_{i+1}^h+x_{i+2}^h\cos\phi_{i+1}^h}{\sin\phi_{i+1}^h}
-
\frac{x_i^h+x_{i+1}^h\cos\phi_i^h}{\sin\phi_i^h},
\]
then, uniformly on such cells,
\[
\boxed{
\frac{D_i^h}{h}
=
\left(\frac{\kappa}{\tau}\right)'(m_i)+O(h)
},
\qquad m_i=s_i+\frac h2.
\]
Thus the finite-step conservation law approaches $(\kappa/\tau)'=0$, equivalently \(\tau/\kappa=\mathrm{constant}\) on intervals where \(\kappa\tau\neq0\). The scalar quotient itself is second-order accurate, whereas its normalized finite difference is first-order accurate. If the underlying smooth curve is a general helix, then
\[
D_i^h=O(h^3).
\]
\end{proposition}

\begin{proof}
Put \(m_i=s_i+h/2\). Symmetric Taylor expansion of the chord about \(m_i\) gives
\[
\gamma(m_i+h/2)-\gamma(m_i-h/2)
=hT(m_i)+\frac{h^3}{24}T''(m_i)+O(h^5).
\]
Since
\[
T''=-\kappa^2T+\kappa'N+\kappa\tau B,
\]
normalization yields
\begin{equation}
T_i^h
=
T(m_i)+\frac{h^2}{24}
\bigl(\kappa'N+\kappa\tau B\bigr)(m_i)+O(h^4).
\label{eq:chord-expansion}
\end{equation}
The neighboring chord tangents are centered at \(s_i-h/2\) and
\(s_i+h/2\). Expanding them symmetrically about \(s_i\) gives
\[
T_i^h-T_{i-1}^h=h\kappa(s_i)N(s_i)+O(h^3).
\]
Because the two vectors are unit,
\[
\|T_i^h-T_{i-1}^h\|
=2\sin\frac{\theta_i^h}{2},
\]
and therefore
\begin{equation}
x_i^h=\tan\frac{\theta_i^h}{2}
=\frac h2\kappa(s_i)+O(h^3).
\label{eq:x-asymptotic}
\end{equation}
Keeping one further term in the same symmetric Taylor calculation gives
\[
x_i^h=\frac h2\kappa(s_i)+h^3A_x(s_i)+O(h^4)
\]
for a smooth coefficient $A_x$. Its explicit expression is not needed below; only the smooth dependence of this coefficient on the cell is used. The $C^6$ hypothesis supplies uniform control of the remainder in this refinement and in the $O(h^4)$ torsion-angle expansion below.

We next compute the binormal more precisely.  Expanding the two adjacent
chord tangents about \(s_i\), taking their cross product, and normalizing
gives
\begin{equation}
B_i^h
=
B(s_i)+h^2C_B(s_i)+O(h^4),
\label{eq:B-asymptotic}
\end{equation}
where
\begin{equation}
C_B
=
-\frac{\kappa\tau}{6}T
-\frac{2\kappa'\tau+\kappa\tau'}{12\kappa}N.
\label{eq:CB-explicit}
\end{equation}
Notice in particular that \(C_B\perp B\), as required by the unit-length
constraint to this order.

To extract the signed torsion angle, write \(m=m_i\) and expand
\(B_i^h\) and \(B_{i+1}^h\) symmetrically about \(m\).  Using
\(B'=-\tau N\), the Frenet equations, \eqref{eq:chord-expansion},
and \eqref{eq:CB-explicit}, one obtains
\begin{align}
-\bigl\langle B_{i+1}^h,N_i^h\bigr\rangle
&=
h\tau(m)+h^3E(m)+O(h^4),
\label{eq:sinphi-expanded}\\
\bigl\langle B_{i+1}^h,B_i^h\bigr\rangle
&=
1-\frac{h^2}{2}\tau(m)^2+O(h^4),
\label{eq:cosphi-expanded}
\end{align}
where
\[
E=
\frac{\kappa^2\tau}{8}-\frac{\tau^3}{6}
+\frac{\tau''}{8}
+\frac{\kappa'\tau'}{6\kappa}
+\frac{\kappa''\tau}{6\kappa}
-\frac{(\kappa')^2\tau}{6\kappa^2}.
\]
Since our sign convention is
\[
\sin\phi_i^h=-\langle B_{i+1}^h,N_i^h\rangle,
\qquad
\cos\phi_i^h=\langle B_{i+1}^h,B_i^h\rangle,
\]
the \(\operatorname{atan2}\) expansion gives
\begin{equation}
\phi_i^h
=
h\tau(m_i)+h^3A_\phi(m_i)+O(h^4),
\label{eq:phi-refined}
\end{equation}
with
\[
A_\phi=
\frac{\kappa^2\tau}{8}
+\frac{\tau''}{8}
+\frac{\kappa'\tau'}{6\kappa}
+\frac{\kappa''\tau}{6\kappa}
-\frac{(\kappa')^2\tau}{6\kappa^2}.
\]
In particular,
\begin{equation}
\phi_i^h=h\tau(m_i)+O(h^3).
\label{eq:phi-asymptotic}
\end{equation}
The estimates \eqref{eq:x-asymptotic} and \eqref{eq:phi-asymptotic}, together with $\kappa\ge\kappa_0>0$ and $|\tau|\ge\tau_0>0$, imply after decreasing $h_0$ if necessary that $0<\theta_i^h<\pi$ and $\phi_i^h\notin\{0,\pi\}$ uniformly on the stated compact cells.

Consequently,
\[
\sin\phi_i^h=h\tau(m_i)+O(h^3),
\qquad
\cos\phi_i^h=1+O(h^2).
\]
Together with the refined $x_i^h$ expansion above and \eqref{eq:phi-refined}, symmetric midpoint expansion gives smooth functions $K$ and $L$ such that
\[
x_i^h+x_{i+1}^h\cos\phi_i^h
=
h\kappa(m_i)+h^3K(m_i)+O(h^4),
\]
and
\[
\sin\phi_i^h
=
h\tau(m_i)+h^3L(m_i)+O(h^4).
\]
Since \(|\tau|\ge\tau_0>0\), division yields
\begin{equation}
q_i^h:=
\frac{x_i^h+x_{i+1}^h\cos\phi_i^h}{\sin\phi_i^h}
=
\rho(m_i)+h^2S(m_i)+O(h^3),
\qquad
\rho=\frac{\kappa}{\tau},
\label{eq:q-refined}
\end{equation}
where
\[
S=\frac{K}{\tau}-\frac{\kappa L}{\tau^2}
\]
is smooth. In particular, $q_i^h=(\kappa/\tau)(m_i)+O(h^2)$, proving the stated second-order quotient consistency.

Since $m_{i+1}=m_i+h$,
\[
D_i^h=q_{i+1}^h-q_i^h
=\rho(m_i+h)-\rho(m_i)+O(h^2),
\]
and Taylor expansion yields
\[
\frac{D_i^h}{h}=\rho'(m_i)+O(h).
\]
If the underlying curve is a smooth general helix, then $\rho$ is constant and the refined expansion \eqref{eq:q-refined} gives
\[
D_i^h
=h^2\bigl(S(m_i+h)-S(m_i)\bigr)+O(h^3)
=O(h^3).
\]
Hence every smooth general helix satisfying the hypotheses has $D_i^h=O(h^3)$.
\end{proof}

\begin{corollary}[Second-order convergence of the locally reconstructed helical direction]\label{cor:axis-smooth-convergence}
Assume the hypotheses of Proposition~\ref{prop:smooth-consistency} and, in addition, that $\gamma$ is a smooth general helix on $J$. Orient its smooth helical direction $U$ so that $T\cdot U>0$, and write
\[
\rho=\frac{\kappa}{\tau},
\qquad
W=T+\rho B=\sec\alpha\,U,
\]
where $\rho$ is constant on $J$. For sufficiently small $h$, define the local binormal coefficient
\[
\widehat y_i^h
=
\frac{x_{i+1}^h+x_i^h\cos\phi_i^h}{\sin\phi_i^h}
\]
and the local reconstructed vector
\[
\widehat Q_i^h
=
T_i^h-x_i^hN_i^h+\widehat y_i^hB_i^h.
\]
Then, uniformly on compact cells contained in $J$,
\[
\widehat Q_i^h=W+O(h^2).
\]
Thus, with
\[
\widehat U_i^h=\frac{\widehat Q_i^h}{\|\widehat Q_i^h\|},
\qquad
\cos\widehat\alpha_i^h=\frac1{\|\widehat Q_i^h\|},
\]
one has
\[
\boxed{\ \|\widehat U_i^h-U\|=O(h^2),\qquad
\cos\widehat\alpha_i^h-\cos\alpha=O(h^2).\ }
\]
Moreover, the finite-step defect of Proposition~\ref{prop:smooth-consistency} satisfies $D_i^h=O(h^3)$ for a smooth general helix.
\end{corollary}

\begin{proof}
Let
\[
q_i^h=\frac{x_i^h+x_{i+1}^h\cos\phi_i^h}{\sin\phi_i^h}.
\]
Proposition~\ref{prop:smooth-consistency} gives $q_i^h=\rho+O(h^2)$. Since
\[
\widehat y_i^h-q_i^h
=
\frac{(x_{i+1}^h-x_i^h)(1-\cos\phi_i^h)}{\sin\phi_i^h},
\]
and $x_{i+1}^h-x_i^h=O(h^2)$, $1-\cos\phi_i^h=O(h^2)$, and $\sin\phi_i^h=O(h)$ with $|\sin\phi_i^h|\ge c h$ on compact cells for small $h$, we obtain
\[
\widehat y_i^h=\rho+O(h^2).
\]
Next, \eqref{eq:B-asymptotic} gives $B_i^h=B(s_i)+O(h^2)$. Also \eqref{eq:chord-expansion}, expanded from $m_i=s_i+h/2$ back to $s_i$, yields
\[
T_i^h=T(s_i)+\frac h2\kappa(s_i)N(s_i)+O(h^2).
\]
Since $N_i^h=B_i^h\times T_i^h$,
\[
N_i^h=N(s_i)-\frac h2\kappa(s_i)T(s_i)+O(h^2).
\]
Together with $x_i^h=\frac h2\kappa(s_i)+O(h^3)$, these expansions give the cancellation
\[
T_i^h-x_i^hN_i^h=T(s_i)+O(h^2).
\]
Therefore
\[
\widehat Q_i^h
=T(s_i)+\rho B(s_i)+O(h^2)
=W+O(h^2).
\]
For a smooth general helix, $W'=\rho'B=0$, so $W$ is constant. Normalization is smooth near the nonzero vector $W$, which proves the two stated $O(h^2)$ reconstruction estimates. The last assertion, $D_i^h=O(h^3)$, follows from the refined estimate in Proposition~\ref{prop:smooth-consistency}.
\end{proof}

\begin{remark}[Why the second-order helical-direction accuracy is helix-specific]\label{rem:axis-order}
The second-order statements depend on uniform arclength sampling and on
the chord-tangent/midpoint convention used above. They also use the helix condition. For a general smooth curve, with $\rho=\kappa/\tau$, the same computation gives
\[
\widehat y_i^h=\rho(m_i)+O(h^2)
=\rho(s_i)+\frac h2\rho'(s_i)+O(h^2),
\]
and hence
\[
\widehat Q_i^h
=\bigl(T+\rho B\bigr)(s_i)
+\frac h2\rho'(s_i)B(s_i)+O(h^2).
\]
The quotient is centered at a midpoint, whereas the discrete frame is based at a vertex. For a general smooth curve this gives a nonzero first-order term. In Corollary~\ref{cor:axis-smooth-convergence}, the $O(h^2)$ estimate follows from $\rho'\equiv0$. The vectors $\widehat Q_i^h$ are local reconstructions and need not be independent of $i$ for a fixed $h$. Thus a chordal sample of a smooth general helix need not itself be an exact discrete general helix. Figure~\ref{fig:smooth-consistency} compares the two orders of convergence.
\end{remark}

\begin{figure}[!ht]
\centering
\resizebox{0.98\textwidth}{!}{\input{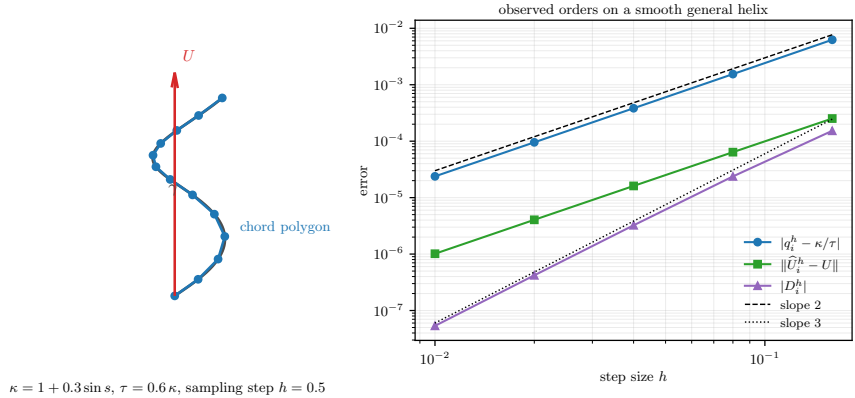}}
\caption{Smooth consistency, Proposition~\ref{prop:smooth-consistency} and Corollary~\ref{cor:axis-smooth-convergence}. Left: a smooth general helix with $\kappa=1+0.3\sin s$ and $\tau=0.6\,\kappa$, so that $\kappa/\tau$ is constant while $\kappa$ and $\tau$ are not, together with its uniform chord polygon for $h=0.5$ and the smooth helical direction $U$. Right: measured errors against $h$. The quotient $q_i^h$ and the reconstructed helical direction $\widehat U_i^h$ converge at order two, and the two reference lines have slopes $2$ and $3$. For every smooth general helix under the stated sampling assumptions, the refined quotient expansion gives $D_i^h=O(h^3)$; the slope-$3$ reference line shows this cubic rate and is not special to this example.}
\label{fig:smooth-consistency}
\end{figure}

\section{Conclusion}

For the non-orthogonal case, we use the vector
\[
Q_i=T_i-\tan(\theta_i/2)N_i+y_iB_i.
\]
Proposition~\ref{prop:conserved-vector} shows that the curve is a discrete general helix if and only if the coefficients $y_i$ can be chosen so that $Q_i$ is constant. Proposition~\ref{prop:spherical-characterization} gives the corresponding spherical description: the tangent indicatrix lies in a plane section of $S^2$. When $\phi_i\notin\{0,\pi\}$, eliminating these coefficients gives the $n-3$ equations of Theorem~\ref{thm:finite-step}. Theorems~\ref{thm:phase-parametrization} and~\ref{thm:linear-subspace} also cover the case $\phi_i=\pi$. The linear theorem says that the half-angle vector belongs to a torsion-dependent subspace of dimension at most two. Remark~\ref{rem:linear-test} gives a direct numerical test based on this criterion.

These formulas also determine the helical direction and angle and give the bound in Corollary~\ref{cor:turning-bound}. The helical direction is not always unique. Proposition~\ref{prop:torsion-realizability} shows that prescribed turning angles and prescribed torsion angles behave differently: arbitrary turning angles can be realized after choosing suitable torsion angles, whereas prescribed torsion angles may fail to be realizable. The edge lengths do not enter the exact condition. Example~\ref{ex:nonconstant-helix} has both nonconstant turning angles and nonconstant torsion angles. The orthogonal case is planar.

We also compare the present Frenet convention with the cross-ratio construction of M\"uller and Vaxman~\cite{MullerVaxman2021}. Their discrete curvature circle uses four consecutive points and is M\"obius invariant, whereas their discrete torsion for space curves is not M\"obius invariant. Here a Frenet frame and its turning angle are determined by three consecutive vertices. The condition $T_i\cdot U=c$ is linear in the Frenet-frame coefficients, which allows the coefficients $y_i$ to be eliminated.

The smooth calculation gives a consistency check for the discrete formulas. Under uniform arclength sampling,
\[
\frac{x_i^h+x_{i+1}^h\cos\phi_i^h}{\sin\phi_i^h}
=
\frac{\kappa}{\tau}\left(s_i+\frac h2\right)+O(h^2),
\]
and the normalized finite-step defect converges to $(\kappa/\tau)'$. If the smooth curve is a general helix, the locally reconstructed helical direction converges with order two and $D_i^h=O(h^3)$.

We assume $0<\theta_i<\pi$, since the chosen Frenet frame is not defined when consecutive tangents are parallel or antiparallel. The quotient formula also excludes $\phi_i=\pi$ because $\sin\phi_i$ occurs in the denominator. The phase, linear, and conserved-vector formulations still include this case. Remark~\ref{rem:cross-multiplication} shows why simple cross multiplication does not remove the restriction. Other discretizations of curvature and torsion may give different formulas.

\section*{Statements and Declarations}

\noindent\textbf{Funding.}
No funding was received to assist with the preparation of this manuscript.

\medskip
\noindent\textbf{Competing interests.}
The authors declare that they have no competing interests relevant to the content of this article.

\medskip
\noindent\textbf{Data availability.}
No datasets were generated or analysed during the current study.

\medskip
\noindent\textbf{Code availability.}
No custom software or code is required to support the mathematical results of this study.

\bigskip
\begin{sloppypar}

\end{sloppypar}

\end{document}